\documentclass{gAPA2e-ppn}

\usepackage{amsmath,amsfonts,amssymb,amsxtra}
\usepackage[colorlinks=true]{hyperref}
\hypersetup{urlcolor=blue, citecolor=red, linkcolor=blue}
\usepackage[usenames,dvipsnames]{color}

\theoremstyle{plain}
\newtheorem{theorem}{Theorem}[section]
\newtheorem{corollary}[theorem]{Corollary}
\newtheorem{lemma}[theorem]{Lemma}
\newtheorem{proposition}[theorem]{Proposition}

\theoremstyle{remark}

\theoremstyle{definition}

\newcommand{\be}[1]{\begin{equation}\label{#1}}
\newcommand{\ee}{\end{equation}}
\renewcommand{\(}{\left(}
\renewcommand{\)}{\right)}
\newcommand{\R}{{\mathbb R}}
\newcommand{\N}{{\mathbb N}}

\newcommand{\ird}[1]{\int_{\R^d}{#1}\,dx}
\newcommand{\nrm}[2]{\left\|{#1}\right\|_{\mathrm L^{#2}(\R^d)}}

\begin{document}

\title{\itshape Flows and functional inequalities for fractional operators}

\author{Jean Dolbeault$^{\ast,{\rm a}}$\thanks{$^{\rm a}$Jean Dolbeault, Email: \href{mailto:dolbeaul@ceremade.dauphine.fr}{dolbeaul@ceremade.dauphine.fr}\vspace{1pt}} and An Zhang$^{\ast,{\rm b}}$\thanks{$^{\rm b}$Corresponding author: An Zhang, Email: \href{mailto:zhang@ceremade.dauphine.fr}{zhang@ceremade.dauphine.fr}\vspace{6pt}}\\
\vspace{6pt} $^{\ast}${\em{CEREMADE (CNRS UMR n$^\circ$ 7534), PSL research university, Universit\'e Paris-Dauphine, Place de Lattre de Tassigny, 75775 Paris 16, France}}}

\maketitle

\begin{abstract}
This paper collects results concerning global rates and large time asymptotics of a fractional fast diffusion on the Euclidean space, which is deeply related with a family of fractional Gagliardo-Nirenberg-Sobolev inequalities. Generically, self-similar solutions are not optimal for the Gagliardo-Nirenberg-Sobolev inequalities, in strong contrast with usual standard fast diffusion equations based on non-fractional operators. Various aspects of the stability of the self-similar solutions and of the entropy methods like \emph{carr\'e du champ} and \emph{R\'enyi entropy powers} methods are investigated and raise a number of open problems.
\end{abstract}


\begin{classcode} Primary: 26A33, 35R11; Secondary: 35A23, 35B33, 35B40, 35K65.


\end{classcode}

\section{Introduction}\label{Sec:Intro}

Recently many papers have been devoted to the extension of results involving second order partial differential operators to fractional operators. More specifically, fractional diffusion equations have been studied from the point of view of existence, comparison and regularity of the solutions, not only in the linear case but also in the framework of nonlinear models of porous medium or fast diffusion type. Concerning modeling issues, fractional diffusions are usually motivated by microscopic jump processes. One can refer to~\cite{bouchaud1990anomalous,metzler2000random} for an extended review of models arising from various areas of physics. We will not go in this direction. Relying on known theoretical results like the ones of~\cite{MR3082241,MR2847534,MR2773189,MR2954615,MR3205647}, our approach aims at the description of fundamental qualitative properties of such  equations in relation with closely associated, basic functional inequalities.

Standard fast diffusion or porous medium equations have simple features which arise from the homogeneity of the nonlinear term or from the fact that the diffusion operator is of second-order. These features explain the special role of self-similar solutions, known as Barenblatt-Pattle, in the large time regimes: see~\cite{MR3191976} for an historical presentation. Also remarkable is the fact discovered in~\cite{MR1940370} that the Barenblatt-Pattle solutions are optimal for some Gagliardo-Nirenberg-Sobolev inequalities, which are essential to measure the asymptotic stability of the self-similar solutions using relative entropy methods. Entropy functionals are indeed deeply connected with fast diffusion equations, as these equations can be seen as gradient flows of the entropies with respect to the Wasserstein's distance (see~\cite{MR1842429}).

In the context of fractional fast diffusion or porous medium equations, it is therefore very natural to question the role of self-similar solutions in terms of large time asymptotics and related functional inequalities. However, fractional order derivatives and nonlinearities do not combine well and this raises a number of difficulties. Self-similar solutions which generalize the Barenblatt-Pattle solutions to fractional diffusion have been studied in connection with large time asymptotics in~\cite{MR2847534,MR2773189} and~\cite{MR2817383,MR3294409,MR3279352}. There is a certain flexibility in the generalization of the standard nonlinear diffusion equations: this has been investigated in~\cite{MR3239623} and~\cite{MR3334174}.

In the present paper, we shall specifically rely on the fractional fast diffusion equation
\[
\frac{\partial u}{\partial t}=\nabla\cdot\(\sqrt u\,\nabla(-\Delta)^{-s}\,u^{m-\frac12}\)
\]
with $m<1$, which seems particularly well adapted to entropy -- entropy production inequalities as we shall see below. Such an equation has already been considered in~\cite[Equation~(MG)]{MR3334174} from the point of view of self-similarity. Our purpose is to clarify the interplay of this equation with fractional Gagliardo-Nirenberg-Sobolev inequalities (see Section~\ref{Sec:FFDE}), and observe that, generically, optimal functions for these inequalities differ from self-similar solutions that are supposed to govern the large time behavior of the solutions of the evolution equation (Proposition~\ref{Prop:SelfSimEL}). Beyond some results about, \emph{e.g.}, the existence of optimal solutions for the interpolation inequalities (Proposition~\ref{Prop:GNS}), we raise a number of open questions concerning the large time asymptotics in Section~\ref{Sec:As} and the applicability of the Bakry-Emery, or \emph{carr\'e du champ}, method in Section~\ref{Sec:BE}.

\section{Preliminaries: a fractional interpolation inequality and a fractional fast diffusion flow}\label{Sec:Prelim}

\subsection{The fractional Sobolev inequality}\label{Sec:Sobolev}

According to~\cite{MR3179693}, for any $\alpha\in(0,d)$, with $q=\frac{2\,d}{d-\alpha}$, the fractional Sobolev inequality in $\R^d$ can be written as
\be{Sobolev}
\|w\|_{\mathring{\mathrm H}^\frac\alpha2(\R^d)}^2\ge\mathsf S_{d,\alpha}\(\ird{|w|^q}\)^\frac2q\quad\forall\,w\in\mathring{\mathrm H}^\frac\alpha2(\R^d)
\ee
where $\mathring{\mathrm H}^\frac\alpha2(\R^d)$ is the space of all tempered distributions $w$ such that
\[
\hat w\in\mathrm L^1_{\mathrm{loc}}(\R^d)\quad\mbox{and}\quad\|w\|_{\mathring{\mathrm H}^\frac\alpha2(\R^d)}^2:=\int_{\R^d}|\xi|^\alpha|\,\hat w|^2\,d\xi<\infty\,.
\]
Here $\hat w$ denotes the Fourier transform of $w$ and the optimal constant is given by
\[
\mathsf S_{d,\alpha}=2^\alpha\,\pi^\frac\alpha2\,\tfrac{\Gamma(\frac{d+\alpha}2)}{\Gamma(\frac{d-\alpha}2)}\,\Bigl(\tfrac{\Gamma(\frac d2)}{\Gamma(d)}\Bigr)^{\alpha/d}\,.
\]
Up to translations, dilations and multiplications by a nonzero constant, the optimal function is
\[
w_\star(x)=\(1+|x|^2\)^{-\frac{d-\alpha}2}\quad\forall\,x\in\R^d\,.
\]
It is easy to check that $w_\star$ solves
\be{Eqn:AT}
(-\Delta)^\frac\alpha2w_\star=\mathsf C_{d,\alpha}\,w_\star^\frac{d+\alpha}{d-\alpha}\quad\mbox{where}\quad\mathsf C_{d,\alpha}=2^\alpha\,\tfrac{\Gamma(\frac{d+\alpha}2)}{\Gamma(\frac{d-\alpha}2)}\,.
\ee
With the notation $\mathcal D_\alpha:=\nabla\,(-\Delta)^\frac{\alpha-1}2$, we get that $(-\Delta)^\frac\alpha2=\mathcal D_{\alpha/2}^*\,\mathcal D_{\alpha/2}$ and~\eqref{Sobolev} can then be written as
\[
\|w\|_{\mathring{\mathrm H}^\frac\alpha2(\R^d)}^2=\nrm{\mathcal D_{\alpha/2}w}2^2\ge\mathsf S_{d,\alpha}\,\nrm wq^2\quad\forall\,w\in\mathring{\mathrm H}^\frac\alpha2(\R^d)\,.
\]
The case $\alpha=2$ was established in~\cite{MR0407905,Talenti-76}, but one may refer to~\cite{MR0289739} in case $d=3$ and even to~\cite{bliss1930integral} for some early considerations on optimal functions. In the fractional case $\alpha\neq2$, the dual form of the optimal inequality, \emph{i.e.}, the Hardy-Littlewood-Sobolev inequality, was established in~\cite{MR717827}, while considerations on the case \hbox{$\alpha=1$} in connection with trace issues can be found in~\cite{MR962929}. Later work include~\cite{MR1307010} and~\cite{MR2737789} among many other related papers. The issue of the symmetry of the optimal functions, up to translations, was considered in~\cite{MR717827} and established in~\cite{MR2200258} using the moving planes method.

\subsection{A fractional Gagliardo-Nirenberg-Sobolev inequality}

With $q=\frac{2\,d}{d-\alpha}$, a simple H\"older interpolation shows that
\[
\nrm w{2p}\le\nrm w{p+1}^{1-\vartheta}\,\nrm wq^\vartheta\quad\mbox{with}\quad\vartheta=\frac dp\,\frac{p-1}{d+\alpha-p\,(d-\alpha)}\,.
\]
Together with Sobolev's inequality, this provides an interpolation inequality of Gagliardo-Nirenberg-Sobolev type. More precisely, we have the following statement.
\begin{proposition}\label{Prop:GNS} Assume that $\alpha\in(0,d)$ and $p\in(1,q]$ with $q=\frac{2\,d}{d-\alpha}$. With the above notations, the following Gagliardo-Nirenberg-Sobolev inequality
\be{GNS}
\nrm w{2p}\le\mathsf C_{\mathrm{GNS}}\,\|w\|_{\mathring{\mathrm H}^\frac\alpha2(\R^d)}^\vartheta\,\nrm w{p+1}^{1-\vartheta}\quad\forall\,w\in\mathcal D(\R^d)
\ee
holds with an optimal constant $\mathsf C_{\mathrm{GNS}}\le\mathsf S_{d,\alpha}^{\vartheta/2}$. Moreover, there exists a positive optimal function in $\mathring{\mathrm H}^\frac\alpha2(\R^d)\cap\mathrm L^{p+1}(\R^d)$ which solves the Euler-Lagrange equation
\be{EL}
(-\Delta)^\frac\alpha2w+w^p-w^{2\,p-1}=0\,,\quad x\in\R^d\,.
\ee
\end{proposition}
In the above statement $\mathcal D(\R^d)$ denotes the space of smooth functions with compact support, but it is simple to extend it by density to the space $\big\{w\in\mathring{\mathrm H}^\frac\alpha2(\R^d)\,:\,\ird{w^{p+1}}<\infty\big\}$. See~\cite[Theorem~2.1]{MR2961251} for details. When $p=\frac q2=\frac d{d-\alpha}$, then $\vartheta=1$ and $\mathsf C_{\mathrm{GNS}}=\mathsf S_{d,\alpha}$. The symmetry of the optimal functions is not considered here and the interested reader is invited to refer for instance to~\cite{doi:10.1080/00036811.2014.898273} for a related problem.

\begin{proof}[Sketch of the proof] This result is out of the scope of the present paper and the proof is relatively standard, so let us give only a simple sketch. A reader interested in further details is invited to refer to~\cite{Trabelsi} and~\cite{Hajaiej201317} for further details in similar cases. The above result can actually be seen as a special case of those dealt with in~\cite{MR2961251}. The convergence of sequences in the critical case has been analyzed in details in~\cite{MR1632171}. We can therefore assume that $p\in(1,q)$.

An optimal function for~\eqref{GNS} can be obtained by considering
\[
\mathcal I_{\mathsf M}:=\inf\left\{\|w\|_{\mathring{\mathrm H}^\frac\alpha2(\R^d)}^2+\nrm w{p+1}^{p+1}\,:\,\ird{w^{2p}}=\mathsf M\right\}\,.
\]
An optimisation on $\lambda>0$ of $w_\lambda=\lambda^\frac d{2p}\,w(\lambda\cdot)$ shows that the minimization of $\mathcal I_{\mathsf M}$ is equivalent to the characterization of the optimality case in~\eqref{GNS}, and also that
\[
\mathcal I_{\mathsf M}=\mathsf M^\gamma\,\mathcal I_1\quad\mbox{with}\quad\gamma=\frac{d+\alpha-p\,(d-\alpha)}{d-p\,(d-2\,\alpha)}<1\,.
\]

With these preliminary observations, a concentration-compactness analysis can be done as follows. Let us consider a sequence $(w_n)_{n\in\N}$ of nonnegative function such that $\ird{w_n^{2p}}=\mathsf M$ for any $n\in\N$ and
\[
\lim_{n\to\infty}\|w_n\|_{\mathring{\mathrm H}^\frac\alpha2(\R^d)}^2+\nrm{w_n}{p+1}^{p+1}=\mathcal I_{\mathsf M}\,.
\]
1) The following result has been stated in~\cite[Lemma I.1, p. 231]{MR778974}.
\begin{lemma}\label{Lem:Lions} Let $R>0$ and $p\in[1,2^*/2)$. If $(w_n)_n$ is bounded in $\mathrm H^1(\R^d)$ and if
\[
\limsup_{n\to\infty}\int_{B_R(x)}|w_n|^{2p}\,dy=0
\]
for any $x\in\R^d$, then $\lim_{n\to\infty}\nrm{w_n}r=0$ for any $r\in[2,2^*)$.\end{lemma}
Here $2^*=\infty$ if $d=1$ or $2$, and $2^*=2\,d/(d-2)$ if $d\ge3$. By an analogue of this result in $\mathring{\mathrm H}^\frac\alpha2(\R^d)$, with $2^*/2$ replaced by $q$, \emph{vanishing} cannot occur, which means that, after eventually replacing $w_n$ by $w_n(\cdot+y_n)$ for some unbounded sequence $(y_n)_{n\in\N}$ of points in $\R^d$, we have
\[
\lim_{R\to\infty}\lim_{n\to\infty}\int_{B_R(0)}|w_n|^{2p}\,dy>0\,.
\]
2) Since $\gamma<1$, it is straightforward to check that
\[
\mathcal I_{\mathsf M_1+\mathsf M_2}\le\mathcal I_{\mathsf M_1}+\mathcal I_{\mathsf M_2}\,,
\]
which prevents \emph{dichotomy.} As a consequence, for any $\varepsilon>0$, there exists some $R>0$, large enough, such that
\[
\lim_{n\to\infty}\int_{B_R(0)}|w_n|^{2p}\,dy\ge\mathsf M-\varepsilon\,.
\]
In other words, this means that the sequence $(w_n^{2p})_{n\in\N}$ is tightly relatively compact in the sense of measures.
\\
3) \emph{Concentration} is forbidden by Sobolev's inequality~\eqref{Sobolev}. Altogether, this proves that $(w_n)_{n\in\N}$ is weakly relatively compact in $\mathrm L^{2p}(\R^d)$ (see for instance~\cite[Lemma~2.2]{Autuori:2013qr}) and the conclusion holds by lower semi-continuity.

Let us take the logarithm of both sides in~\eqref{GNS}. An optimal function is a minimizer of the functional
\[
w\mapsto\vartheta\,\log\|w\|_{\mathring{\mathrm H}^\frac\alpha2(\R^d)}+(1-\vartheta)\,\log\nrm w{p+1}-\log\,\nrm w{2p}\,.
\]
A variation shows that
\[
\frac{\vartheta\,(-\Delta)^\frac\alpha2w}{\|w\|_{\mathring{\mathrm H}^\frac\alpha2(\R^d)}}+\frac{(1-\vartheta)\,w^p}{\nrm w{p+1}}-\frac{w^{2\,p-1}}{\nrm w{2p}}=0\,.
\]
Using the homogeneity in~\eqref{GNS} and a multiplication by a constant, the coefficients can be adjusted so that $w$ solves~\eqref{EL}.\end{proof}

\subsection{A fractional fast diffusion equation}\label{Sec:FFDE}

Let us consider the case
\[
\alpha=2\,(1-s)\,.
\]
Any nonnegative solution $u$ of
\be{Eqn}
\frac{\partial u}{\partial t}=\nabla\cdot\(\sqrt u\,\nabla(-\Delta)^{-s}\,w\)\quad\mbox{with}\quad w=u^{m-\frac12}
\ee
which is smooth enough and has good decay properties as $|x|\to\infty$ is such that
\be{EI}
\mathsf E'=\tfrac{2\,m\,(1-m)}{2\,m-1}\,\mathsf I\,,
\ee
where the \emph{entropy} $\mathsf E$ and the \emph{Fisher information} $\mathsf I$ are defined respectively by
\[
\mathsf E[u]:=\ird{u^m}\quad\mbox{and}\quad\mathsf I[u]:=\ird{\left|\nabla^{(1-s)}u^{m-\frac12}\right|^2}\,,
\]
and $\nabla^{(1-s)}$ is defined by
\[
\nabla^{(1-s)}w:=\nabla(-\Delta)^{-s/2}w\,.
\]
Above we consider $\mathsf E[u(t,\cdot)]$ as a function of $t$ if $u$ solves~\eqref{Eqn} and $\mathsf E'$ denotes the $t$-derivative of $\mathsf E[u(t,\cdot)]$.

Because the right hand side in~\eqref{Eqn} is in divergence form, we know that
\[
\mathsf M:=\ird u
\]
is independent of $t\ge0$ for any solution $u$ which is smooth enough, and with enough decay as $|x|\to+\infty$. Let us assume that
\[
m\ge m_1:=\frac12+\frac1q=1-\frac\alpha{2\,d}\,.
\]
If $m=m_1$ and $v_\star:=w_\star^q$, then $v_\star(x)=\(1+|x|^2\)^{-d}$ satisfies
\[
v_\star^{m_1-\frac12}=w_\star\quad\mbox{and}\quad(-\Delta)^{1-s}\(v_\star^{m_1-\frac12}\)=\mathsf C_{d,\alpha}\,v_\star^\frac{d+\alpha}{2\,d}=\mathsf C_{d,\alpha}\,\(1+|x|^2\)^{s-1}\,\sqrt{v_\star}
\]
according to~\eqref{Eqn:AT}.

Assume now that $m\in[m_1,1)$. If we introduce a \emph{generalized R\'enyi entropy power functional}~$\mathsf F$ as in~\cite{MR3200617,1501}, defined by
\[
\mathsf F:=\mathsf E^\sigma
\]
with
\[
\sigma:=\frac2\vartheta\,\frac{1-\vartheta}{p+1}+1=\frac{m-m_c}{1-m}\,,\quad p=\frac1{2\,m-1}\quad\mbox{and}\quad m_c:=1-\frac\alpha d\,,
\]
then by applying Proposition~\ref{Prop:GNS} to $w=u^{m-\frac12}$ we obtain that
\[
\mathsf F'\ge\kappa
\]
using~\eqref{EI}, where the constant $\kappa$ depends only on $d$, $s$ (or $\alpha$), $m$ and $\mathsf M$, and can be computed as
\[
\kappa:=\frac{2\,m\,(1-m)}{2\,m-1}\,\sigma\(\frac{\mathsf M}{\mathsf C_{\mathrm{GNS}}^{2p}}\)^\frac1{p\,\vartheta}=\frac{2\,m}{2\,m-1}\,(m-m_c)\(\frac{\mathsf M}{\mathsf C_{\mathrm{GNS}}^{2p}}\)^\frac{d+\alpha-p\,(d-\alpha)}{d\,(p-1)}\,.
\]

\subsection{Self-similar solutions}

Let us introduce the time-dependent change of variables
\be{TDRS}
u(t,x)=\frac1{R^d}\,v\(\log R,\frac xR\)
\ee
where $R=R(t)$ is given by the ordinary differential equation
\be{ODE}
R^{\mu-1}\,\frac{dR}{dt}=1\,,\quad R(0)=1\,,
\ee
and the parameter $\mu$ is defined by
\[
\mu:=d\,(m-m_c)\quad\mbox{with}\quad m_c=1-\frac\alpha d\,.
\]
If $m$ is in the range $[m_1,1)$, it is straightforward to check that $\mu$ is positive and the scale $R$ has the explicit expression
\[
R(t)=\(1+\mu\,t\)^{1/\mu}\quad\forall\,t\ge0\,.
\]
If $u$ solves~\eqref{Eqn}, it is then easy to check that $v$ is a solution of
\be{RescaledEqn}
\frac{\partial v}{\partial t}=\nabla\cdot\left[\sqrt v\(\nabla(-\Delta)^{-s}\,v^{m-\frac12}+x\,\sqrt v\)\right]\,.
\ee

This change of variable is classical in case of non-fractional operators but has also been considered in the case of the fractional Laplacian, for instance in~\cite{MR2817383,MR3294409}. A detailed analysis of self-similar profiles is available in~\cite{MR3334174}. In the fast diffusion case, for an equation related with~\eqref{Eqn}, we refer to~\cite{MR3298369}, where the role of Barenblatt type solutions for regularization effects has been clarified.
\begin{proposition}\label{Prop:SelfSimEL} Let $p=\frac1{2\,m-1}$. Under the assumptions of Proposition~\ref{Prop:GNS}, if $w$ is an optimal function for~\eqref{GNS}, then $v=w^{2p}$ is not a bounded, positive, stationary solution of~\eqref{RescaledEqn} such that $\mathsf I[v]$ is finite, unless $s=0$, $\alpha=2$ or $m=\frac{d+1}{d+2\,(1-s)}$.\end{proposition}
In other words, we claim that~\eqref{EL} does not provide stationary solutions of~\eqref{RescaledEqn} as it is the case in the non-fractional case, except in the special case corresponding to $m=\frac{d+1}{d+2\,(1-s)}$, which was observed in~\cite{MR3239623,MR3334174}.

\begin{proof} We argue by contradiction and consider a solution of
\[
\nabla\cdot\left[\sqrt v\(\nabla(-\Delta)^{-s}\,v^{m-\frac12}+x\,\sqrt v\)\right]=0\,.
\]
We deduce that $w=v^{m-\frac12}$ solves
\[
0=\nabla(-\Delta)^{-s}\,v^{m-\frac12}+x\,\sqrt v=\nabla(-\Delta)^{-s}\,w+x\,w^p
\]
where $p=1/(2\,m-1)$ and, as a consequence,
\[
-(-\Delta)^{1-s}\,w+\nabla\cdot(x\,w^p)=0\,.
\]
With $\alpha=2\,(1-s)$, if $w$ simultaneously solves~\eqref{EL} up to a multiplication by a constant, a scaling and a translation, this means that $w^p=v$ is such that
\[
\nabla\cdot\(x\,w^p\)=a\,w^{2\,p-1}-b\,w^p
\]
for some positive constants $a$ and $b$ that can be adjusted. With
\[
a=\frac p{p-1}\quad\mbox{and}\quad b=\frac p{p-1}-d\,,
\]
we find that
\[
w(x)=\(1+|x|^2\)^\frac1{1-p}\quad\forall\,x\in\R^d\,.
\]
The result follows from~\cite[Theorem~3.1]{MR3239623}.\end{proof}

The case $m=m_1$ is almost explicit. In order to illustrate Proposition~\ref{Prop:SelfSimEL}, let us give some details. With the notations of Section~\ref{Sec:Sobolev}, if we choose
\[
v_\star^{m_1-\frac12}=w_\star\quad\Longleftrightarrow\quad v_\star=w_\star^\frac{2\,d}{d-\alpha}\quad\Longleftrightarrow\quad v_\star(x)=\(1+|x|^2\)^{-d}\quad\forall\,x\in\R^d\,,
\]
we obtain that
\begin{multline*}
\nabla\,(-\Delta)^{-1}\,(-\Delta)^{1-s}\,v_\star^{m_1-\frac12}+x\,\sqrt v_\star=\nabla\,(-\Delta)^{-1}\,\mathsf C_{d,\alpha}\,w_\star^\frac{d+\alpha}{d-\alpha}+x\,\sqrt{v_\star}\\
=\nabla\,(-\Delta)^{-1}\(\mathsf C_{d,\alpha}\,v_\star^\frac{d+\alpha}{2\,d}-\mathsf C_{d,2}\,v_\star^\frac{d+2}{2\,d}\)\neq 0
\end{multline*}
unless $\alpha=2$.

\medskip For later purpose, let us define the self-similar profile $\mathfrak B$ as the unique radial solution of
\[
\nabla(-\Delta)^{-s}\,\mathfrak B^{m-\frac12}+x\,\sqrt{\mathfrak B}=0
\]
such that $\ird{\mathfrak B}=\mathsf M$. We recall that these solutions are the so-called \emph{Barenblatt profiles} when $s=0$ and refer to~\cite{BBDGV} for more details. It is straightforward to check that
\[
u_\star(t,x)=\frac1{R^d}\,\mathfrak B\(\log R,\frac xR\)
\]
is a self-similar solution of~\eqref{Eqn} if $R=R(t)$ is given by~\eqref{ODE}.

\section{Global rates of convergence and formal large time asymptotics}\label{Sec:As}

The \emph{generalized R\'enyi entropy power functional}
defined as in Section~\ref{Sec:FFDE} by
\[
\mathsf F[u]=\(\ird{u^m}\)^\sigma
\]
is such that
\[
\mathsf F[u(t,\cdot)]\ge\mathsf F[u_0]+\kappa\,t\quad\forall\,t\ge0
\]
if $u$ solves~\eqref{Eqn} with initial datum $u_0$ on the one hand, and a direct computation shows that the self-similar solution $u_\star$ satisfies the relation
\[
\mathsf F[u_\star(t,\cdot)]=\mathsf F[\mathfrak B]\,(1+\mu\,t)\quad\forall\,t\ge0\,.
\]
\begin{corollary}\label{Cor:SelfSimEstim} With the above notations and under the assumptions of Proposition~\ref{Prop:SelfSimEL}, we have the estimate
\[
\kappa<\kappa_\star:=\mu\,\mathsf F[\mathfrak B]\,.
\]
\end{corollary}
\begin{proof} The inequality $\kappa\le\kappa_\star$ is a straightforward consequence of the above computations, with $u_0=\mathfrak B$. The strict inequality follows from Proposition~\ref{Prop:SelfSimEL}.\end{proof}

Now let us investigate the large time asymptotics at formal level. Inspired by~\cite{MR3298369,MR3191976}, we expect that for any $m\in[m_1,1)$,
\[
u(t,x)\sim u_\star(t,x)\quad\mbox{as}\quad t\to+\infty
\]
as in~\cite[Theorem~1.4]{MR3485132}, and thus
\[
\mathsf F[u(t,\cdot)]=\kappa_\star\,t\,\big(1+o(1)\big)=\mathsf F[\mathfrak B]\,(1+\mu\,t)\,\big(1+o(1)\big)\,.
\]
Using the self-similar change variables~\eqref{TDRS}-\eqref{ODE}, the equivalence $u\sim u_\star$ amounts to
\[
v(t,\cdot)\to\mathfrak B\quad\mbox{as}\quad t\to+\infty\,.
\]
So far this question is open. Next let us investigate at formal level the question of the attraction of the solutions by the self-similar solutions.

If $u$ solves~\eqref{Eqn} and $u_\star$ is the self-similar solution with same mass, we define the relative entropy by
\[
\mathcal E[u]=\tfrac1{m-1}\ird{\(u^m-\,u_\star^m-\,m\,u_\star^{m-1}\,(u-u_\star)\)}\,.
\]
A straightforward computation shows that
\begin{multline*}
\frac d{dt}\mathcal E[u(t,\cdot)]=\tfrac m{m-1}\ird{\(u^{m-1}-\,u_\star^{m-1}\)\frac{\partial u}{\partial t}}-\,m\ird{u_\star^{m-2}\(u-\,u_\star\)\frac{\partial u_\star}{\partial t}}\\
\hspace*{2cm}=-\tfrac m{m-1}\ird{\nabla\(u^{m-1}-\,u_\star^{m-1}\)\cdot\(\sqrt u\,\nabla(-\Delta)^{-s}\,u^{m-\frac12}\)}\\
+\,m\ird{\nabla\(u_\star^{m-2}\(u-\,u_\star\)\)\cdot\(\sqrt{u_\star}\,\nabla(-\Delta)^{-s}\,u_\star^{m-\frac12}\)}\,.
\end{multline*}
At formal level, we investigate the limit as $\varepsilon\to0$ of
\[
u_\varepsilon=u_\star\(1+\varepsilon\,u_\star^{\frac12-m}\,f\)\,.
\]
First, let us notice that
\begin{multline*}
\sqrt{u_\varepsilon}=\sqrt{u_\star}+\tfrac12\,\varepsilon\,u_\star^{1-m}\,f+o(\varepsilon)\,,\quad u_\varepsilon^{m-\frac12}=u_\star^{m-\frac12}+\(m-\tfrac12\)\,\varepsilon\,f+o(\varepsilon)\,,\\
u_\varepsilon^{m-1}=u_\star^{m-1}+(m-1)\,\varepsilon\,\frac f{\sqrt{u_\star}}+\tfrac12\,(m-1)\,(m-2)\,\varepsilon\,\frac{f^2}{u_\star^m}+o(\varepsilon^2)\,.
\end{multline*}
As a consequence, we find that
\[
\mathcal E[u]=\tfrac m2\,\varepsilon^2\ird{u_\star^{1-m}\,f^2}+o(\varepsilon^2)
\]
and
\[
\frac d{dt}\mathcal E[u(t,\cdot)]=-\,\tfrac m2\,\varepsilon^2\,\mathcal Q_{u_\star}[f]+o(\varepsilon^2)
\]
where the quadratic form $\mathcal Q_{u_\star}$ is defined by
\begin{multline*}
\mathcal Q_{u_\star}[f]:=\ird{\!\!\nabla\big(\tfrac f{\sqrt{u_\star}}\big)\cdot\((2\,m-1)\,\sqrt{u_\star}\,\nabla(-\Delta)^{-s}\,f+\,u_\star^{1-m}\,f\,\nabla(-\Delta)^{-s}\,u_\star^{m-\frac12}\)}\\
+(m-2)\ird{\!\!\nabla\(\frac{f^2}{u_\star^m}\)\cdot\(\sqrt{u_\star}\,\nabla(-\Delta)^{-s}\,u_\star^{m-\frac12}\)}\,.
\end{multline*}

If $s=0$, we observe that
\[
\tfrac1{2\,m-1}\,\mathcal Q_{u_\star}[f]=\ird{u_\star\,\left|\nabla\big(\tfrac f{\sqrt{u_\star}}\big)\right|^2}-\tfrac{1-m}2\ird{\((2-m)\,\frac{|\nabla u_\star|^2}{u_\star^2}-\frac{\Delta u_\star}{u_\star}\)f^2}\,.
\]
To characterize the asymptotic stability of $u_\star$, the point is to notice that $\mathcal Q_{u_\star}[f]$ controls $\ird{u_\star^{1-m}\,f^2}$.
\begin{proposition}\label{Prop:Gap} With the previous notations, if $s=0$ and $m_1\le m<1$, there exists a positive constant $\mathcal C$ which depends on only on $m$ and $d$ such that
\[
\mathcal Q_{u_\star}[f]\ge\frac2{1-m}\,\frac{\mathcal C}{R^\mu}\ird{u_\star^{1-m}\,f^2}
\]
for any smooth function $f$ such that $\ird{u_\star^{\frac32-m}\,f}=0$.\end{proposition}
We recall that $R=R(t)$ is given by~\eqref{ODE}, and $u_\star$ depends on $t$, so that the above inequality holds for any $t\ge0$. When $s=0$, $\mu=d\,m-(d-2)$ and $R^\mu(t)=1+\mu\,t$ for any $t\ge0$. By density, it is easy to extend the above inequality to any function $f\in\mathrm L^\infty(dt,\mathrm L^2(\R^d,u_\star^{1-m}\,dx))$ such that $\mathcal Q_{u_\star}[f]$ is finite for any $t\ge0$.
\begin{proof} The result follows from the change of variables~\eqref{TDRS} and~\cite[Corollary~1]{1004}: the above inequality can be rewritten in the form of a spectral gap inequality with a constant which is independent of $t$. Also see~\cite{BBDGV} for earlier spectral gap computations. \end{proof}
In self-similar variables, Proposition~\ref{Prop:Gap} shows that $\mathfrak B$ is linearly stable, and this is enough to prove that $u_\star$ is asymptotically linearly stable when $s=0$. Actually one also knows that $\mathfrak B$ is globally stable when $s=0$, with explicit global rates, which is a far deeper issue. Proposition~\ref{Prop:Gap} immediately raises two open questions:
\begin{enumerate}
\item[(i)] Is $u_\star$ asymptotically linearly stable for any $s\in(0,n)$ ?
\item[(ii)] Can we prove that $u_\star$ is, under the appropriate mass normalization, a global attractor and deduce that the asymptotic rate of convergence of any solution towards $u_\star$ is determined by the spectral gap $\Lambda$ associated with the inequality
\[
\mathcal Q_{\mathfrak B}[g]\ge\Lambda\ird{\mathfrak B^{1-m}\,g^2}
\]
for any admissible function $g$ such that $\ird{\mathfrak B^{\frac32-m}\,g}=0$ ?
\end{enumerate}

\section{On the Bakry-Emery method applied to the fractional fast diffusion equation and generalized R\'enyi entropy powers}\label{Sec:BE}

In this section we do a computation based on the generalized R\'enyi entropy powers in case $s=0$ and emphasize the points which differ when $s>0$. So far such a computation has been done using the so-called \emph{pressure} variable $u^{m-1}$ as, for instance, in~\cite{MR3200617,1501}. Here we base our computations on
\[
f:=u^{m-\frac12}\,.
\]
With this simple change of variables, we obtain that
\[
\frac{\partial f}{\partial t}=u^{m-1}\,\left[\(m-\tfrac12\)\,\Delta f+\tfrac12\,\frac{|\nabla f|^2}f\right]
\]
and compute the time derivative of the Fisher information
\[
\mathsf I[u]=\ird{\left|\nabla f\right|^2}
\]
as
\[
\mathsf I'=-\,2\ird{(\Delta f)\,\frac{\partial f}{\partial t}}=-\,(2\,m-1)\ird{u^{m-1}\,(\Delta f)^2}-\ird{u^{m-1}\,\Delta f\,\frac{|\nabla f|^2}f}\,.
\]
The computation of the r.h.s.~relies on two identities. Here we assume that integrations by parts can be taken without any special precautions. A detailed discussion of this issue can be found in~\cite{dolbeault:hal-01395771}.

\medskip\noindent$\bullet$ \emph{First identity:} we use an integration by parts to obtain
\begin{multline*}
-\ird{u^{m-1}\,\Delta f\,\frac{|\nabla f|^2}f}\\
=\ird{\nabla f\cdot\nabla\(u^{m-1}\,\frac{|\nabla f|^2}f\)}\hspace*{6cm}\\
=\ird{(\nabla u^{m-1})\cdot\nabla f\,\frac{|\nabla f|^2}f}+2\ird{u^{m-1}\,\mathrm Hf:\frac{\nabla f\otimes\nabla f}f}\\
-\ird{u^{m-1}\,\frac{|\nabla f|^4}{f^2}}
\end{multline*}
where $\mathrm Hf$ denotes the Hessian matrix of $f$. After integrating by parts again, we obtain
\be{Id1}
-\ird{u^{m-1}\,\Delta f\,\frac{|\nabla f|^2}f}=2\ird{u^{m-1}\,\mathrm Hf:\frac{\nabla f\otimes\nabla f}f}-\tfrac1{2\,m-1}\ird{u^{m-1}\,\frac{|\nabla f|^4}{f^2}}
\ee
using
\be{Eqn:uv}
\nabla\(u^{m-1}\)=\tfrac{2\,(m-1)}{2\,m-1}\,u^{m-1}\,\frac{\nabla f}f\,.
\ee

\noindent$\bullet$ \emph{Second identity:} using again~\eqref{Eqn:uv}, we find that
\begin{multline*}
-\,(2\,m-1)\ird{u^{m-1}\,(\Delta f)^2}=(2\,m-1)\ird{\nabla f\cdot\nabla\(u^{m-1}\,\Delta f\)}\\
=2\,(m-1)\ird{u^{m-1}\,\Delta f\,\frac{|\nabla f|^2}f}+(2\,m-1)\ird{u^{m-1}\,\nabla f\cdot\nabla\Delta f}\,.
\end{multline*}
Next we use the identity $\nabla f\cdot\nabla\Delta f=\tfrac12\,\Delta\big(|\nabla f|^2\big)-\|\mathrm Hf\|^2$ and obtain
\begin{multline*}
-\,(2\,m-1)\ird{u^{m-1}\,(\Delta f)^2}\\
=2\,(m-1)\ird{u^{m-1}\,\Delta f\,\frac{|\nabla f|^2}f}-\,(2\,m-1)\ird{u^{m-1}\,\|\mathrm Hf\|^2}\\
+\(m-\tfrac12\)\ird{u^{m-1}\,\Delta\big(|\nabla f|^2\big)}
\end{multline*}
where, using twice~\eqref{Eqn:uv}, we obtain that
\begin{multline*}
\ird{u^{m-1}\,\Delta\big(|\nabla f|^2\big)}=\tfrac{2\,(m-1)}{2\,m-1}\ird{|\nabla f|^2\,\nabla\cdot\(u^{m-1}\,\frac{\nabla f}f\)}\\
=\tfrac{2\,(m-1)}{2\,m-1}\ird{u^{m-1}\,\Delta f\,\frac{|\nabla f|^2}f}+\tfrac{2\,(m-1)}{2\,m-1}\,\left[\tfrac{2\,(m-1)}{2\,m-1}-1\right]\ird{u^{m-1}\,\frac{|\nabla f|^4}{f^2}}\\
=\tfrac{2\,(m-1)}{2\,m-1}\ird{u^{m-1}\,\Delta f\,\frac{|\nabla f|^2}f}-\tfrac{2\,(m-1)}{(2\,m-1)^2}\ird{u^{m-1}\,\frac{|\nabla f|^4}{f^2}}\,.
\end{multline*}
Notice that the above computation will cause trouble for a generalization to the case of a function $f$ which is defined as a non-local expression of $u^{m-\frac12}$. Hence
\begin{multline*}
-\,(2\,m-1)\ird{u^{m-1}\,(\Delta f)^2}\\
=3\,(m-1)\ird{u^{m-1}\,\Delta f\,\frac{|\nabla f|^2}f}-\,(2\,m-1)\ird{u^{m-1}\,\|\mathrm Hf\|^2}\\
-\tfrac{m-1}{2\,m-1}\ird{u^{m-1}\,\frac{|\nabla f|^4}{f^2}}\,.
\end{multline*}
We can evaluate $\ird{u^{m-1}\,\Delta f\,\frac{|\nabla f|^2}f}$ using~\eqref{Id1} and get that
\begin{multline*}
-\,(2\,m-1)\ird{u^{m-1}\,(\Delta f)^2}\\
=-\,3\,(m-1)\(2\ird{u^{m-1}\,\mathrm Hf:\frac{\nabla f\otimes\nabla f}f}-\tfrac1{2\,m-1}\ird{u^{m-1}\,\frac{|\nabla f|^4}{f^2}}\)\\
-\,(2\,m-1)\ird{u^{m-1}\,\|\mathrm Hf\|^2}-\tfrac{m-1}{2\,m-1}\ird{u^{m-1}\,\frac{|\nabla f|^4}{f^2}}\,.
\end{multline*}
This provides the second identity
\begin{multline}\label{Id2}
-\,(2\,m-1)\ird{u^{m-1}\,(\Delta f)^2}\\=-\,(2\,m-1)\ird{u^{m-1}\,\|\mathrm Hf\|^2}-\,6\,(m-1)\ird{u^{m-1}\,\mathrm Hf:\frac{\nabla f\otimes\nabla f}f}\\
+\tfrac{2\,(m-1)}{2\,m-1}\ird{u^{m-1}\,\frac{|\nabla f|^4}{f^2}}\,.
\end{multline}

Collecting terms of our two identities~\eqref{Id1} and~\eqref{Id2}, we have found that
\begin{multline*}
\mathsf I'=-\,(2\,m-1)\ird{u^{m-1}\,\|\mathrm Hf\|^2}-\,2\,(3\,m-4)\ird{u^{m-1}\,\mathrm Hf:\frac{\nabla f\otimes\nabla f}f}\\
-\,\tfrac{3-2\,m}{2\,m-1}\ird{u^{m-1}\,\frac{|\nabla f|^4}{f^2}}\,,
\end{multline*}
that we can also write as
\[
\mathsf I'=-\,(2\,m-1)\(1-\tfrac{m_1}m\)\ird{u^{m-1}\,\left|\Delta f-\tfrac1{2\,m-1}\,\frac{|\nabla f|^2}f\right|^2}-\mathcal R
\]
with
\begin{multline*}
\mathcal R:=(2\,m-1)\,\frac{m_1}m\ird{u^{m-1}\,\|\mathrm Hf\|^2}\\
-\,2\(3\,(1-m_1)+\tfrac{m_1}m\)\ird{u^{m-1}\,\mathrm Hf:\frac{\nabla f\otimes\nabla f}f}\\
+\,\tfrac{2\,m\,(1-m_1)+m_1}{m\,(2\,m-1)}\ird{u^{m-1}\,\frac{|\nabla f|^4}{f^2}}\,.
\end{multline*}
Recalling that $m_1=(d-1)/d$, the reader is invited to check that
\[
\mathcal R=\tfrac{2\,m-1}m\ird{u^{m-1}\left\|\mathrm Lf-\tfrac1{2\,m-1}\,\mathrm Mf\right\|^2}
\]
using the notations
\[
\mathrm Lf:=\mathrm Hf-\tfrac1d\,\Delta f\,\mathrm{Id}\quad\mbox{and}\quad\mathrm Mf:=\frac{\nabla f\otimes\nabla f}f-\tfrac1d\,\frac{|\nabla f|^2}f\,\mathrm{Id}\,,
\]
so that
\[
\mathrm Hf:\frac{\nabla f\otimes\nabla f}f=\mathrm Hf:\mathrm Mf+\tfrac1d\,\Delta f\,\frac{|\nabla f|^2}f=\mathrm Lf:\mathrm Mf+\tfrac1d\,\Delta f\,\frac{|\nabla f|^2}f
\]
and
\[
\|\mathrm Lf\|^2=\|\mathrm Hf\|^2-\tfrac1d\,(\Delta f)^2\quad\mbox{and}\quad\|\mathrm Mf\|^2=\(1-\tfrac1d\)\frac{|\nabla f|^4}{f^2}\,.
\]

Finally we introduce the R\'enyi entropy power $\mathsf F$ defined by
\[
\mathsf F:=\mathsf E^\sigma\quad\mbox{with}\quad\sigma:=\frac 2d\,\frac1{1-m}-1\,.
\]
If $1-\frac1d=m_1\le m<1$, proving that, as a function of $t$, $\mathsf F$ is concave amounts to proving that
\[
\mathsf G:=\tfrac{2\,m-1}{2\,m\,(1-m)}\,\mathsf E^{2-\sigma}\,\mathsf F''=(\sigma-1)\,\tfrac{2\,m\,(1-m)}{2\,m-1}\,\mathsf I^2+\mathsf E\,\mathsf I'
\]
is nonpositive, because of~\eqref{EI}. We observe that
\begin{multline*}
\ird{u^{m-1}\,\left|\Delta f-\tfrac1{2\,m-1}\,\frac{|\nabla f|^2}f+\,\tfrac{2\,m}{2\,m-1}\,\frac{\mathsf I}{\mathsf E}\,\sqrt u\right|^2}\\
=\ird{u^{m-1}\,\left|\Delta f-\tfrac1{2\,m-1}\,\frac{|\nabla f|^2}f\right|^2}-\big(\tfrac{2\,m}{2\,m-1}\big)^2\,\frac{\mathsf I^2}{\mathsf E}
\end{multline*}
and, as a consequence,
\begin{multline*}
\mathsf G=-\,(2\,m-1)\(1-\frac{m_1}m\)\mathsf E\ird{u^{m-1}\,\left|\Delta f-\tfrac1{2\,m-1}\,\frac{|\nabla f|^2}f+\,\tfrac{2\,m}{2\,m-1}\,\frac{\mathsf I}{\mathsf E}\,\sqrt u\right|^2}\\
-\,\tfrac{2\,m-1}m\,\mathsf E\ird{u^{m-1}\left\|\mathrm Lf-\tfrac1{2\,m-1}\,\mathrm Mf\right\|^2}\,.
\end{multline*}

Summarizing, when $s=0$ and by assuming that there are no boundary terms in the integrations by parts, we establishes that $\mathsf F$ is a concave increasing function, and shows that
\[
\mathsf F'\ge\kappa=\kappa_\star
\]
with the notations of Sections~\ref{Sec:FFDE} and Corollary~\ref{Cor:SelfSimEstim}. According to~\cite{dolbeault:hal-01395771}, this is equivalent to the proof of Inequality~\eqref{GNS}. The main difference with~\cite{dolbeault:hal-01395771} is that the computations are done in terms of $f=u^{m-\frac12}$ instead of $u^m$.

When $s\neq0$, it would be desirable to do a similar computation with
\[
f=(-\Delta)^{-s/2}\,u^{m-\frac12}
\]
but there are several obstructions. First of all, according to Proposition~\ref{Prop:SelfSimEL}, we cannot expect that $\kappa=\kappa_\star$. In other words, the self-similar solution $u_\star$ is not optimal for Inequality~\eqref{GNS}. We can hope that monotonicity of $\mathsf F'$ holds for some non optimal constant as a consequence of the Bakry-Emery method applied to the fractional fast diffusion equation and generalized R\'enyi entropy powers: proving it is still open. Technically another obstruction arises from the computations as we have no identity equivalent to~\eqref{Eqn:uv}. As in~\cite{MR2817383,MR3294409} in the porous medium case, one could expect that a Stroock-Varadhan identity~\cite{MR1409835,MR0387812,MR0400425} allows to give a proof of the optimality with non-optimal constants, and also characterize the stability of the self-similar solutions as $t\to+\infty$, but this is so far also an open question.

\par\medskip\noindent{\small{\bf Acknowledgements.} This work is supported by a public grant overseen by the French National Research Agency (ANR) as part of the ``Investissements d'Avenir'' program (A.Z., reference: ANR-10-LABX-0098, LabEx SMP) and by the projects \emph{STAB} (J.D., A.Z.) and \emph{Kibord} (J.D.) of the French National Research Agency (ANR). A.Z.~thanks the ERC Advanced Grant BLOWDISOL (Blow-up, dispersion and solitons; PI: Frank Merle) \# 291214 for support. The authors thank Maria J.~Esteban for fruitful discussions and suggestions.
\par\smallskip\noindent\copyright~2016 by the authors. This paper may be reproduced, in its entirety, for non-commercial purposes.}

\end{document}